\documentclass[12pt]{amsart}

\newtheorem{theorem}{Theorem}[section]
\newtheorem{lemma}[theorem]{Lemma}

\theoremstyle{definition}
\newtheorem{definition}[theorem]{Definition}

\theoremstyle{remark}
\newtheorem{remark}[theorem]{Remark}

\numberwithin{equation}{section}

\usepackage{amsmath}
\usepackage[english]{babel}
\usepackage{babelbib}
\usepackage{natbib}
\RequirePackage{geometry}
\usepackage[utf8]{inputenc}
\usepackage[T1]{fontenc}
\setcitestyle{square,numbers}

\usepackage{comment}
\usepackage{float}
\usepackage{etoolbox}
\usepackage{txfonts}

\usepackage[pdftex]{graphicx}

\usepackage{listingsutf8}

\usepackage{geometry}
\usepackage[usenames,dvipsnames,svgnames,table]{xcolor}
\usepackage[mathcal]{euscript}
\usepackage{mathrsfs}
\usepackage{setspace}
\newcommand{\Zi}{\mathbb{Z}[i]}
\newcommand{\ds}{\displaystyle}

\newcounter{ctProba}
\DeclareMathAlphabet\mathcal{OMS}{cmsy}{m}{n}

\newcommand{\R}{\mathbb{R}}
\newcommand{\Z}{\mathbb{Z}}

\newcommand{\N}{\mathbb{N}}
\newcommand{\Q}{\mathbb{Q}}
\newcommand{\OK}{\mathcal{O}_K}

\makeatletter
\newcommand{\neutralize}[1]{\expandafter\let\csname c@#1\endcsname\count@}
\makeatother

\theoremstyle{plain}

\theoremstyle{plain}
\theoremstyle{plain}
\newtheorem{proposition}[theorem]{Proposition}
\theoremstyle{plain}
\newtheorem{cor}[theorem]{Corollary}

\theoremstyle{definition}

\theoremstyle{remark}

\DeclareMathOperator{\Rez}{Re}
\DeclareMathOperator{\Imz}{Im}

\makeatletter
\newtheorem*{rep@theorem}{\rep@title}
\newcommand{\newreptheorem}[2]{%
\newenvironment{rep#1}[1]{%
 \def\rep@title{\normalfont{\textbf{#2 \ref{##1}}}}%
 \begin{rep@theorem}}%
 {\end{rep@theorem}}}
\makeatother

\newreptheorem{theorem}{Theorem}

\begin{document}
\newgeometry{top=5.5cm, left=2cm, right=1.8cm}
\title{On the size of Diophantine $m$-tuples in imaginary quadratic number rings}

\author{Nikola Ad\v zaga}

\subjclass[2010]{primary 11D09; secondary 11J68} 

\date{July 4, 2018}

\keywords{Diophantine $m$-tuples, Diophantine approximation, Pell equations, Diophantine equations}

\begin{abstract}
A Diophantine $m$-tuple is a set of $m$ distinct integers such that the product of any two distinct elements plus one is a perfect square. It was recently proven that there is no Diophantine quintuple in positive integers. We study the same problem in the rings of integers of imaginary quadratic fields. By using a gap principle proven by Diophantine approximations, we show that  $m\leqslant 42$. Our proof is relatively simple compared to the proofs of the similar results in positive integers.
\end{abstract}
\maketitle
\setstretch{1.15}

\section{Introduction}
\label{intro}
A long-standing conjecture, motivated by work of Baker and Davenport \cite{baker}, that there is no Diophantine quintuple, was recently proven \cite{nemapetorke}. At a similar time, in \cite{sestorke}, it was found that there are infinitely many Diophantine sextuples in rational numbers. Generalizing results from integers to rationals can be very hard. On the other hand, it is reasonable to assume that the transition to analogous problems in other rings of integers will be easier. However, even in this ambient, there are not many results. E.g.~we find less than $10$ papers solving similar problems in the ring of Gaussian integers. We can highlight \cite{BFT} and \cite{Zrinka}, which deal with the extension of Diophantine triples from one-parameter families. Yet, due to the similarities of this ring with the usual ring of rational integers, we can expect that the similar results will hold and that one can prove them using similar techniques. The goal of this research is to answer how many elements can a Diophantine $m$-tuple in imaginary quadratic number ring have, or, to be more specific and modest, to provide an upper bound on the size of a Diophantine $m$-tuple.

Let $\OK$ be the ring of integers of an imaginary quadratic number field $K$. A Diophantine triple $\{a, b, c\}\subseteq \OK$ induces an elliptic curve $E: y^2=(ax+1)(bx+1)(cx+1)$ defined over $K$. For every $d$ such that $\{a, b, c, d\}$ is Diophantine quadruple, there exist $r, s, t$ such that $ad+1=r^2$, $bd+1=s^2$ and $cd+1=t^2$. We obtain an $\OK$-integral point $(d, rst)$ on $E$. Since an elliptic curve has finitely many integral points \cite{siegel}, every Diophantine triple can be extended with only finitely many elements.  We conclude that every Diophantine $m$-tuple is finite. Unfortunately, this does not provide us with an effective absolute bound on the size of such sets. 
 \restoregeometry 
We approach this problem in the following way. Assume that a Diophantine triple $\{a, b, c\}$ in imaginary quadratic number ring $\OK$ can be extended with a fourth element $d$. By eliminating $d$ from the equations it satisfies ($ad+1=x^2$, $bd+1=y^2$ and $cd+1=z^2$), we get a system of two Pell-type equations with common unknown. 
A solution of this system gives us two simultaneous approximations of square roots close to $1$. Then we apply a variant of Bennett's theorem developed for imaginary quadratic number rings by Jadrijević and Ziegler \cite{borka}. We obtain the gap principle which can be stated as follows. In a Diophantine quadruple, if the third element (by size) is much bigger than the second element, then the fourth element is bounded from above by a power of the third element. Combining this result with a simple lower bound on the largest element of a Diophantine quadruple (in terms of the smallest element), we obtain a proof by contradiction (assuming that there exists a Diophantine $m$-tuple with $m\geqslant 43$). Thus, the main result of this paper is the following theorem.

\begin{reptheorem}{thm:g}
\textit{There is no Diophantine $m$-tuple in imaginary quadratic number ring with $m\geqslant 43$.}
\end{reptheorem}

\section{System of Pell-type equations}
\label{sec:Pells}
Let $\{a, b, c\} \subset \OK$ be a Diophantine triple in the imaginary quadratic number ring $\OK$. Without loss of generality, we may assume $0 < |a| \leqslant |b| \leqslant |c|$. Then there are $r, s$ and $t$ in $\OK$ such that $ab+1 = r^2, ac+1=s^2, bc+1=t^2.$
 
Since the equation $X^2-Y^2=1$ has only trivial solutions ($XY=0$), the numbers $ac$ and $bc$ are not squares in $\OK$. Namely, $ac\neq i^2=-1$ and $|s|\neq 1$ since there are no three numbers $\{a, b, c\}$ of absolute value $2$ or less which make up a Diophantine triple. Similarly, $ab$ is not a square: if we assume the contrary, then $r=0$ and $\{a, b\}=\{-1, 1\}$ (or $\{a, b\}=\{\frac{-1+\sqrt{-3}}{2}, \frac{1+\sqrt{-3}}{2}\}$). There is no $c\neq 0$ such that $\{a, b, c\}$ is a Diophantine triple, because $-c+1=s^2$ and $c+1=t^2$ imply $1-c^2=s^2t^2$, hence $c=0$ or $st=0$, which means $c=\pm 1 \in \{a, b\}$. Similar reasoning resolves the case $\{a, b\}=\{\frac{-1+\sqrt{-3}}{2}, \frac{1+\sqrt{-3}}{2}\}$, subtracting and multiplying equations containing $c$ shows that $c=\pm 1$, but $\{a, b, c\}$ is not a Diophantine triple.

Let us note here that the solutions of this equation ($X^2-Y^2=1$) in the ring of integers of imaginary quadratic field are described in \cite{fjel}. Numbers $ab$, $ac$ and $bc$ are not squares even in $K$ since $\OK$ is integrally closed in $K$.
We have proven the following lemma.
\begin{lemma}\label{acnekva} If $\{a, b, c\}$ is a Diophantine triple in the imaginary quadratic number ring $\OK$ and $abc\neq 0$, then $ab$, $ac$ and $bc$ are not squares in $K$.
\end{lemma}
If there is $d\in \OK$ such that $\{a, b, c, d\}$ is a Diophantine quadruple, then there are $x, y, z\in\OK$ such that $ad+1 = x^2, bd+1=y^2, cd+1=z^2$. Therefore $a(cd+1)-c(ad+1)=az^2-cx^2$ and $az^2-cx^2=a-c$. Analogously, eliminating $d$ from the second and the third equation we get $bz^2-cy^2 = b-c$. We have obtained a system of the equations
\begin{align}
az^2-cx^2 &= a-c \label{Eq:Pellac}\\
bz^2-cy^2 &= b-c .\label{Eq:Pellbc}
\end{align}
These equations are similar to Pell's equations and their solutions have a very similar structure. The solutions of Pell-type equations ($x^2-Dy^2=N$) in imaginary quadratic rings are described in \cite{fjel}.

From now on, we will always assume that $0$ is not an element of a Diophantine $m$-tuple. We will also assume that the elements are sorted by absolute value (in ascending order).

\section{Gap principle (obtained by Diophantine approximations)}\label{sec:gap}
Here we prove that the solution of \eqref{Eq:Pellac} and \eqref{Eq:Pellbc} provides us two simultaneous approximations of square roots close to $1$ by elements of $K$.
\begin{lemma}\label{tm:apx}
Let $(x, y, z)$ be a solution of the system of equations \eqref{Eq:Pellac} and \eqref{Eq:Pellbc}. Assume that $|c| > 4|b|$ and $|a|\geqslant 2$. If $\displaystyle \theta_1^{(1)} = \pm\frac{s}{a}\sqrt{\frac{a}{c}}, \theta_1^{(2)} = -\theta_1^{(1)} $ and $\displaystyle \theta_2^{(1)} = \pm\frac{t}{b}\sqrt{\frac{b}{c}}, \theta_2^{(2)} = -\theta_2^{(1)}$, where the signs are chosen such that \[ \left|\theta_1^{(1)} - \frac{sx}{az}\right| \leqslant \left|\theta_1^{(2)} - \frac{sx}{az}\right| \quad \text{ and   }\quad \left|\theta_2^{(1)} - \frac{ty}{bz}\right| \leqslant \left|\theta_1^{(2)} - \frac{ty}{bz}\right|,\] then
\begin{center}\vskip -1.5em
\begin{align*}
\left|\theta_1^{(1)} - \frac{sbx}{abz}\right| \leqslant \frac{|s|\cdot|c-a|}{|a|\sqrt{|ac|}} \cdot \frac{1}{|z|^2} &< \frac{21}{16}\frac{|c|}{|a|}\cdot\frac{1}{|z|^2} \quad\text{ and }\\
\left|\theta_2^{(1)} - \frac{tay}{abz}\right| \leqslant \frac{|s|\cdot|c-b|}{|b|\sqrt{|bc|}} \cdot \frac{1}{|z|^2} &< \frac{21}{16}\frac{|c|}{|a|}\cdot\frac{1}{|z|^2}.
\end{align*}
\end{center}
\end{lemma}

\begin{proof} The following holds
\begin{align*}
\left|\theta_1^{(1)} - \frac{sx}{az}\right| &= \left|(\theta_1^{(1)})^2 - \frac{s^2x^2}{a^2z^2}\right|\cdot \left|\theta_1^{(1)} + \frac{sx}{az}\right|^{-1} = \left|\frac{s^2}{a^2}\right| \cdot \frac{|az^2-cx^2|}{|c||z|^2} \cdot \left|\theta_1^{(2)} - \frac{sx}{az}\right|^{-1}\\
&= \frac{|s|^2 |c-a|}{|a|^2|c|} \left|\theta_1^{(2)} - \frac{sx}{az}\right|^{-1} \cdot \frac{1}{|z|^2}.
\end{align*}

Since 
$\ds 2\left|\theta_1^{(2)} - \frac{sx}{az}\right| \geqslant \left|\theta_1^{(2)} - \frac{sx}{az}\right|+\left|\theta_1^{(1)} - \frac{sx}{az}\right| \geqslant \left|\theta_1^{(2)} - \frac{sx}{az} - \left(\theta_1^{(1)} - \frac{sx}{az}\right)\right| = |\theta_1^{(2)}-\theta_1^{(1)}| = 2 \left| \frac{s}{a}\sqrt{\frac{a}{c}}\right|$, we conclude that $\displaystyle  \left|\theta_1^{(2)} - \frac{sx}{az}\right| \geqslant \left|\frac{s}{a}\sqrt{\frac{a}{c}} \right|$.

Hence $\displaystyle \left|\theta_1^{(1)} - \frac{sbx}{abz}\right| \leqslant  \frac{|s|^2 |c-a|}{|a|^2|c|} \left| \frac{a}{s} \sqrt{\frac{c}{a}}\right| \cdot\frac{1}{|z|^2}$, which implies the first inequality in this lemma's statement. We also need to prove $\displaystyle \frac{|s|\cdot|c-a|}{|a|\sqrt{|ac|}} \cdot \frac{1}{|z|^2} < \frac{21}{16}\frac{|c|}{|a|}\cdot\frac{1}{|z|^2}$, i.e.~$|\sqrt{ac+1}|\cdot|c-a| < \dfrac{21}{16}|c|\sqrt{|ac|}$, which is equivalent to $\displaystyle \left| \sqrt{1+\frac{1}{ac}}\right|< \frac{21}{16}\frac{|c|}{|c-a|}$. The condition $|c| > 4|a|$ implies that $\dfrac{21}{16}\dfrac{|c|}{|c-a|} \geqslant \dfrac{21}{20}$. Left-hand side is
$\displaystyle \left| \sqrt{1+\frac{1}{ac}}\right| \leqslant \sqrt{\left|1+\frac{1}{|ac|}\right|} \leqslant \sqrt{1+\frac{1}{16}} = \frac{\sqrt{17}}{4}$, implying the second inequality stated. 

The other pair of inequalities is proven analogously (we have used $|c|>4|a|$ and $|a|\geqslant 2$). 
\end{proof}

We will now apply Jadrijević--Ziegler theorem from \cite{borka}. 

\newpage \begin{theorem}[{Jadrijević--Ziegler \cite[Theorem 7.1]{borka}}]
Let $\theta_i = \sqrt{1+\frac{a_i}{T}}, i = 1,2$ with $a_1$ and $a_2$ distinct quadratic integers in the imaginary quadratic field $K$ and let $T$ be an algebraic integer of $K$. Further, let $M=\max\{|a_1|,|a_2|\}$, $|T| > M$ and \[ L=\frac{27}{16|a_1|^2|a_2|^2|a_1-a_2|^2}(|T|-M)^2 > 1. \]

\noindent Then \[ \max\left\{\left|\theta_1-\frac{p_1}{q}\right|, \left|\theta_2-\frac{p_2}{q}\right|\right\} > c|q|^{-\lambda},\]
for all algebraic integers $p_1, p_2, q\in K$, where $\ds \lambda = 1+\frac{\log P}{\log L}$, $c^{-1} = 4pP(\max\{1, 2l\})^{\lambda-1}$,

\[ l = \frac{27}{64}\frac{|T|}{|T|-M}, p=\sqrt{\frac{2|T|+3M}{2|T|-2M}},
P=16\frac{|a_1|^2|a_2|^2|a_1-a_2|^2}{\min\{|a_1|, |a_2|, |a_1-a_2|\}^3} (2|T|+3M).\]
\end{theorem}
Using this theorem, we will show that, if the second and the third element of a Diophantine $m$-tuple are sufficiently away from each other, then the fourth element is bounded by a power of the third element. More precisely, we prove the following proposition.

\begin{proposition}[Gap principle]\label{prop:alpha}
If $\{a, b, c, d\} \subseteq \OK$ is a Diophantine quadruple such that $|ac| \geqslant 9$, $\ds |b| \geqslant \frac 32 |a|$, $|b| > 5$ and $|c| > |b|^{15}$, then $|d| < 4278^{20} |c|^{50}$.
\end{proposition}
\begin{proof}
The solution $(x, y, z)$ of \eqref{Eq:Pellac} and \eqref{Eq:Pellbc} gives us simultaneous approximations of \[ \theta_1 = \pm\frac{s}{a}\sqrt{\frac{a}{c}} = \pm \sqrt{\frac{s^2a}{a^2c}} = \pm \sqrt{\frac{ac+1}{ac}} =\pm \sqrt{1+\frac{b}{abc}} \quad \text{ and } \quad \theta_2 =\pm \sqrt{1+\frac{a}{abc}}.\] 

\noindent We let $a_0 = 0, a_1 = b, a_2 = a, T=abc, M=|b|$, and note that $\ds l = \frac{27|abc|}{64(|abc|-|b|)} < \frac 12$,
since this is equivalent to $27|ac|<32(|ac|-1)$, which holds for $\ds |ac| > \frac{32}{5}$. Similarly,
\begin{equation}\label{ineq:p}
p = \sqrt{\frac{2|abc|+3|b|}{2|abc|-2|b|}} = \sqrt{1+\frac{5}{2(|ac|-1)}} \leqslant \sqrt{1+\frac{5}{2(9-1)}} = \sqrt{\frac{21}{16}}.
\end{equation}
Since $l < \dfrac 12$, it follows that $\displaystyle c = \frac{1}{4pP}$. Hence, by \eqref{ineq:p}, $\displaystyle c \geqslant \frac{1}{\sqrt{21}P}$.

\noindent We also observe that $\ds \min\{|a|, |b|, |b-a|\} \geqslant \frac{|a|}{2}$ because $\ds |b-a|\geqslant |b|-|a|\geqslant \frac{|a|}{2}$. Hence
\begin{align*}P &= 16\frac{|a|^2|b|^2|b-a|^2}{ \min\{|a|, |b|, |b-a|\}^3}(2|abc|+3|b|) \leqslant 128 \frac{|b|^3|b-a|^2}{|a|}(2|ac|+3) \text{  and} \\
L &= \frac{27}{16|a|^2|b|^2|b-a|^2}(|abc|-|b|)^2 = \frac{27(|ac|-1)^2}{16|a|^2|b-a|^2}.
\end{align*}

Here, the condition $L > 1$ of Jadrijević-Ziegler theorem is equivalent to $27(|ac|-1)^2 > 16|a|^2|b-a|^2$. Taking the square root, we get a simpler claim $3 \sqrt{3}(|ac|-1) > 4|a||b-a|$. Since $|c| > |b|^3$ (and $|b| \geqslant \frac 32 |a|$), even a stronger claim, $|ac|-1 > |a||b-a|$, holds. Namely, using the assumptions and the triangle inequality we get $|ac|-1 > |a|\cdot|b|^3-1 > 2 |a|^2 |b|-1 > |a|\cdot|b|+|a^2| \geqslant |ab-a^2|=|a||b-a|$.


We observe that $\lambda > 1$ since both $P$ and $L$ are greater than $1$ ($|b|>|a|$ implies $P>1$).
We will show that $|c|>|b|^{15}$ implies $\lambda < 1.9$. Since $\ds \lambda = 1 +\frac{\log P}{\log L}$, we obtain the equivalent statements $\ds\lambda < 1.9$, $\ds \frac{\log P}{\log L} < 0.9$, and $P < L^{0.9}$. Plugging $L$ and using the inequality proven for $P$, we see that we should prove
\[ 128|b|^3 |b-a|^{3.8}|a|^{0.8}(2|ac|+3) < \left(\frac{27}{16}\right)^{0.9} (|ac|-1)^{1.8}.\]

\noindent Since $\ds |ac|-1 >\frac{8}{21}(2|ac|+3)$, it suffices to show that $\ds 336|b|^3|b-a|^{3.8}|a|^{0.8} < \left(\frac{27}{16}\right)^{0.9} (|ac|-1)^{0.8}$. We prove the stronger inequality
\begin{equation}\label{ineq:lmbdokaz}
210|b|^3|b-a|^{3.8}|a|^{0.8} < (|ac|-1)^{0.8}.
\end{equation}

\noindent The right-hand side of \eqref{ineq:lmbdokaz}, $(|ac|-1)^{0.8} $ is greater than $(|ac|)^{0.8}-1$, because the function $f(t) = (t-1)^{0.8}-t^{0.8}+1$ is $0$ at $t=1$ and the function is increasing. Hence, if $|c| > |b|^{15}$, then $(|ac|-1)^{0.8} > |a|^{0.8}|b|^{12}-1$, which is greater than the left-hand side of \eqref{ineq:lmbdokaz}, $210|b|^3|b-a|^{3.8}|a|^{0.8}$, because of the larger degree of $|b|$. More precisely, since $\ds |b|\geqslant \frac 32|a|$, i.e.~$\ds|a|\leqslant \frac 23 |b|$, we conclude that
\[ 210|b|^3|b-a|^{3.8}|a|^{0.8} \leqslant 210|b|^3\left(\frac 53 |b| \right)^{3.8} \left|\frac{2b}{3}\right|^{0.8} < 1058 |b|^{7.6} < |b|^{12}-1.\]

\noindent The last inequality is again obtained by simple analysis of auxiliary function $f(t)=t^{12}-1058t^{7.6}-1$ whose largest root $4.86836$ is determined numerically. Since $|b| > 5$, we have proven that $\lambda < 1.9$.

Jadrijević-Ziegler theorem, together with Lemma \ref{tm:apx}, yields
\[ \frac{21}{16} \frac{|c|}{|a|} \cdot \frac{1}{|z|^2} > \frac{1}{\sqrt{21}P} |abz|^{-\lambda} \geqslant \frac{|a|}{\sqrt{21}\cdot  128|b|^3|b-a|^2(2|ac|+3)} |abz|^{-\lambda},\]
implying $\ds 168\sqrt{21} \frac{|c|}{|a|^2} |b|^3 |b-a|^2 (2|ac|+3)\cdot |ab|^\lambda > |z|^{2-\lambda} > |z|^{0.1}$.
Therefore, \begin{align*}
|z|^{0.1} &< 168\sqrt{21}|c|\cdot 3|ac| \cdot |b-a|^2 |b|^{3+\lambda}|a|^{\lambda-2} \\
&< 504\sqrt{21}|c|^2 \cdot \frac 23 |b| \cdot \left(\frac 53 |b|\right)^2 |b|^{4.9} \quad (\text{since } |b|\geqslant 2|a| \text{ and } \lambda<1.9) \\
&< 4728 |c|^2 |b|^{7.9} < 4728 |c|^{\frac{30+7.9}{15}} < 4728 |c|^{2.53}
\end{align*}

\noindent and, finally, $|z| < 4728^{10} |c|^{2.53\cdot 10} = 4728^{10} |c|^{25.3}$.
\vskip 0.25em
We conclude $\ds
|d| = \frac{|z^2-1|}{|c|} \leqslant \frac{|z|^2+1}{|c|} \leqslant \frac{4728^{20}|c|^{50.6}+1}{|c|} < 4728^{20}|c|^{50}$. 
\end{proof}

\section{A lower bound on the element extending a Diophantine triple}
\label{sec:lower}
In the previous section we have proved an upper bound on the fourth element of a Diophantine $m$-tuple. We will now find a lower bound so that we can juxtapose these two bounds. To find it, we need a definition and a few lemmas.

\begin{definition}A Diophantine triple $\{a, b, c\}$
is \emph{regular} if $c=a+b\pm 2r$, where $r^2=ab+1$.
\end{definition}

\begin{lemma}If $\{a, b, c, d\}$ is a Diophantine quadruple such that $2 \leqslant |a| \leqslant |b| \leqslant |c| \leqslant |d|$, then at least one of the triples $\{a, b, c\}$ and $\{a, b, d\}$ is not regular, i.~e.~it is impossible that $c=a+b-2r$ and $d=a+b+2r$ (or vice versa), where $ab+1=r^2$.
\end{lemma}
\begin{proof}Assume the contrary. Then $cd=(a+b)^2-4r^2=a^2+2ab+b^2-4(ab+1)=a^2-2ab+b^2-4$ and $cd+1=(a-b)^2-3$. Since $\{c, d\}$ is a Diophantine pair, there is an integer $z\in\OK$ such that $cd+1=z^2$. Hence $(a-b)^2-3=z^2$, i.e.~$(a-b-z)(a-b+z)=3$.  Therefore, on the left-hand side there are two elements of $\OK$ of absolute value less than $3$ (or one of them is a unit) and their norms are divisors of $9$.

First we deal with non-unit solutions in the case $K=\Q[\sqrt{D}]$ where $D\equiv 1\pmod{4}$. Let $\rho=\frac{-1+\sqrt{D}}{2}$ be the generator of $\OK$ over $\Z$. Using the fact that the absolute value of the norm of $x+y\rho$ is $\left|x^2+xy-\frac{D-1}{4}y^2\right| = (x+\frac{y}{2})^2+\frac{|D|}{4}y^2=3$, we easily see that for $|D|\geqslant 13$, $y$ must be $0$ and $x+\frac{y}{2}=\pm\sqrt{3}$, which is impossible. For $D=-7$, multiplying the equation by $4$, we get $(2x+y)^2+7y^2=12$, which has no solutions in integers $x$ and $y$. For $D=-3$ we get the equation $(2x+y)^2+3y^2=12$ which has solutions $x=-1, y=2$ and $x=1, y=-2$. This gives us $a-b-z=-1+2\rho, a-b+z=1-2\rho$ (or vice versa) and $cd=z^2-1=-4\sqrt{-3}$ (or $cd=-4$), which implies that $|c| \leqslant \sqrt{4\sqrt{3}}$. Now we check all triples with absolute value of elements between $2$ and $\sqrt{4\sqrt{3}}$ to find only two Diophantine triples $\{-2, 2, -2\sqrt{-3}\}$ and $\{-2, 2, 2\sqrt{-3}\}$. However, $-2\sqrt{-3}\cdot 2\sqrt{-3} + 1=13$ is not a square, so $\{a, b, a+b-2r, a+b+2r\} = \{-2, 2, -2\sqrt{-3}, 2\sqrt{-3}\}$ is not a Diophantine quadruple.

Now let $K=\Q[\sqrt{D}]$ where $D\equiv 2, 3 \pmod{4}$. Analogously as in the previous case, $x^2+|D|y^2 < 3$ has no non-unit solutions for $|D|\geqslant 6$. Similarly, for $D=-2$, we get the solution $a-b\pm z=1\pm \sqrt{-2}$, which implies $2z=\pm 2\sqrt{-2}$, i.e.~$cd=-3$. However, one checks that there is no Diophantine quadruple in $\Z[\sqrt{-2}]$ with such small elements (and $|a|\geqslant 2$).

Unit solutions are easy to check. E.g.~in $\Zi$ number $3$ is prime, and it can be factored in $\Zi$ in the following ways: $3=1\cdot 3, 3=-1\cdot(-3), 3=i\cdot(-3i), 3=-i\cdot 3i$, up to the order of the factors. In first two cases, subtracting $a-b+z$ and $a-b-z$ implies $2z=\pm 2$, i.~e.~$z=\pm 1$, which would imply $cd=0$. In the other two cases, we conclude $z=\pm 2i$ in the same manner, which implies $cd+1=-4$ and $cd=-5$, which is impossible for $|c| > \sqrt{5}$. Now one just checks that there is no Diophantine triple $\{a, b, c\}$ such that $2\leqslant |a|, |b|, |c| \leqslant \sqrt{5}$. Analogous reasoning resolves the number ring $\OK = \Z[\frac{-1+\sqrt{-3}}{2}]$.
\end{proof}

\begin{lemma}\label{Omega} If $\{a, b, c, d\}$ is a Diophantine quadruple such that $2 \leqslant |a| \leqslant |b| \leqslant |c| \leqslant |d|$, then $\ds |d|\geqslant \frac{|ab|}{8}\geqslant \frac{|a|^2}{8}$.
\end{lemma}
\begin{remark}
We conjecture that the stronger claim also holds, if $d\neq a+b+c+2abc\pm 2rst$ (where $s^2=ac+1$ and  $t^2=bc+1$), then $|d|\geqslant 4|ab|$. This claim holds in $\Z$ (see, e.~g.~\cite{Jones}).
\end{remark}
\begin{proof} Using the previous lemma, without loss of generality, we may assume that $\{a, b, d\}$ is not a regular triple. Denote $ab+1=r^2, ad+1=x^2, bd+1=y^2$.

Then $c_{\pm} = a+b+d+2abd\pm2rxy$ is not $0$. Indeed, the claim $c_{\pm}=0$ is equivalent to the claims $a+b+d+2abd = \mp 2rxy$, $(2ab+1)^2d^2+2(a+b)(2ab+1)d+(a+b)^2=4(ab+1)(ad+1)(bd+1)$ and $d^2-2(a+b)d+(a-b)^2-4=0$.
Solutions of this quadratic equation for $d$ are exactly $a+b\pm 2r$. Since $\{a, b, d\}$ is not a regular triple ($d\neq a+b\pm 2r$), we conclude that $c_{\pm} \neq 0$.

Consider now the product $c_+c_- = (a+b+d+2abd)^2-4r^2x^2y^2 = a^2+b^2+d^2-2ab-2ad-2bd-4$.
Hence $|c_+c_-| \leqslant |a|^2+|b|^2+|d|^2+2|ab|+2|ad|+2|bd|+4$. Since $|d|\geqslant |a|, |b|$ and $|d|^2 \geqslant 4$ (there is no Diophantine quadruple with the largest element less than $2$ by absolute value), it follows that $|c_+c_-| \leqslant 10|d|^2$.

On the other hand, $|c_++c_-|=2|a+b+d+2abd|$. Assume that $|c_+|\geqslant |c_-|$. Then $2|c_+|\geqslant |c_+|+|c_-| \geqslant |c_++c_-|$ and $|c_+| \geqslant |a+b+d+2abd|$.

Since $|b|\geqslant|a|\geqslant 2$, it follows that $|ab|\geqslant 4$, so $\ds |a+b+d| \leqslant 3|d| \leqslant \frac 34 |abd|$. We conclude that \[ |c_+| \geqslant |a+b+d+2abd| \geqslant 2|abd|-|a+b+d| \geqslant 2|abd|-\frac 34 |abd| = \frac 54 |abd|.\]

\noindent Juxtaposition of this lower bound on $|c_+|$ with the upper bound on $|c_+c_-|$ implies that $\ds |c_-| \leqslant \frac{10|d|^2}{|c_+|} \leqslant \frac{10|d|^2}{\frac 54 |abd|} = \frac{8|d|}{|ab|}$. Because $c_-\neq 0$, it follows that $|c_|\geqslant 1$, so we can conclude that $\ds \frac{8|d|}{|ab|} \geqslant |c_-| \geqslant 1$ and $\ds |d|\geqslant \frac{|ab|}{8} \geqslant \frac{|a|^2}{8}$. 
\end{proof}

\section{An upper bound on the size of Diophantine $m$-tuple}
\label{sec:main}
Here we prove our main result, the following theorem.
\begin{theorem}\label{thm:g}
There is no Diophantine $m$-tuple in imaginary quadratic number ring $\OK$ with $m\geqslant 43$.
\end{theorem} 
\begin{proof}
Assume the contrary, that there is a Diophantine $m$-tuple $\{a_1, a_2, \dotsc, a_m\}$ sorted by absolute value ($0 < |a_1| \leqslant \dots \leqslant |a_m|$) and $m\geqslant 43$.
By using the computer, we have checked that there is no Diophantine quintuple with absolute value of elements at most $16$. For a fixed $D$ it is clear that one can do this. For $|D| > 32$ all the elements of absolute value less or equal than $16$ are real ones. On the other hand, for $|D|>256$, it is not possible that $a_ia_j+1=(y\sqrt{-D})^2=-Dy^2$ since $|a_ia_j+1|\leqslant 257$.  Therefore, $|a_4|\geqslant 2$, $|a_5|\geqslant 16$. Now we repeatedly apply Lemma \ref{Omega} on different subsets of $\{a_1, a_2, \dotsc, a_m\}$, and each obtained inequality is used in the next appliance:
\begin{align*}
\{a_7, a_8, a_9, a_{10}\}   &\overset{\text{\scriptsize{Lemma }} \ref{Omega}}{\implies} |a_{10}| \geqslant \frac{|a_7|^2}{8}\\
\{a_{10}, a_{11}, a_{12}, a_{13}\}   & {\implies} |a_{13}|  \geqslant \frac{|a_{10}|^2}{8} \geqslant \frac{|a_7|^4}{8^3}\\
\{a_{13}, a_{14}, a_{15}, a_{16}\}   & {\implies} |a_{16}|\geqslant \frac{|a_{13}|^2}{8} \geqslant \frac{|a_7|^8}{8^7}\\
& \quad \vdots\\
\{a_{22}, a_{23}, a_{24}, a_{25}\}  & {\implies}  |a_{25}| \geqslant \frac{|a_7|^{64}}{8^{63}}\\
\end{align*}
\vskip -2em

Let us now show that we can apply the Proposition \ref{prop:alpha} on $\{a_4, a_7, a_{25}, a_{25+k}\}$ for $k>0$ since $|a_4| \geqslant 2$ and $|a_5| > 12$. Namely, by using the Lemma \ref{Omega} on $\{a_4, a_5, a_6, a_7\}$, we can conclude that $|a_7| \geqslant \frac{|a_4a_5|}{8} > \frac 32|a_4|$. Previously obtained inequality guarantees that $|a_{25}| > |a_7|^{15}$ (it suffices to show that $\ds \frac{|a_7|^{64}}{8^{63}} > |a_7|^{15}$, i.~e.~$|a_7| > 14.5$, which holds because $|a_7|\geqslant |a_5| \geqslant 16$). The conditions of the Proposition \ref{prop:alpha} hold.

Therefore, the Proposition \ref{prop:alpha} implies \begin{equation}\label{ineq24} |a_{25+k}| < 4278^{20}|a_{25}|^{50}.\end{equation}

However, we can continue to apply Lemma $\ref{Omega}$:
\begin{align*}
\{a_{25}, a_{26}, a_{27}, a_{28}\} &{\implies} |a_{28}| \geqslant \frac{|a_{25}|^2}{8} \\
\{a_{28}, a_{29}, a_{30}, a_{31}\} &{\implies} |a_{31}| \geqslant \frac{|a_{28}|^2}{8}\geqslant \frac{|a_{25}|^4}{8^3} \\
\dots &{\implies} |a_{43}| \geqslant  \frac{|a_{25}|^{64}}{8^{63}}  > 4278^{20} |a_{25}|^{50},
 \end{align*}
which contradicts the inequality \eqref{ineq24}. Namely, the last inequality is equivalent to $|a_{25}|^{14}\geqslant 8^{63}\cdot 4278^{20}$, i.~e.~it is true for $|a_{25}| > 1.784 \cdot 10^9$, which holds since $\ds |a_{25}| \geqslant \frac{|a_7|^{64}}{8^{63}}\geqslant \frac{16^{64}}{8^{63}}$. 
\end{proof}

Let us note here that computer search did not yield any Diophantine quintuple in imaginary quadratic numbers rings, nor has the more systematic search by P.~E.~Gibbs found Diophantine quintuples in $\Z[\sqrt{-d}]$ for positive integer $d < 50$ (as reported on the ResearchGate).
Even so, we do not see any a priori reason why would all the imaginary quadratic number rings have the largest Diophantine $m$-tuple of the same size $m$. The method used here most likely will not suffice to prove the strongest upper bound, even in the more specific situation of Gaussian integers. However, it is interesting that we have managed to find such a simple proof of the first uniform bound on the size od Diophantine $m$-tuple (in imaginary quadratic number rings) because such results usually required more complex proofs involving linear forms in logarithms (see, for example, \cite{duje}, for the first proof of the uniform upper bound in integers).

\section*{Acknowledgements}
This work was supported by the Croatian Science Foundation under the project no. 6422.

The author would like to thank Andrej Dujella, Alan Filipin and Zrinka Franu\v si\' c for motivating this research and suggesting numerous improvements to both the proof and the manuscript text.

\end{document}